\documentclass[10pt,reqno]{amsart}
\usepackage{amsaddr}
\usepackage{graphicx}
\usepackage{amssymb}
\usepackage{amsfonts}
\usepackage[]{amsmath}
\usepackage[]{epsfig}
\usepackage[]{float}
\usepackage{setspace}
\usepackage{tikz}
\newtheorem{theorem}{Theorem}[section]
\newtheorem{lemma}{Lemma}[section]
\newtheorem{definition}{Definition}[section]
\newtheorem{remark}{Remark}[section]

\numberwithin{equation}{section}
\newcommand{\R}{\mathbb R}

\usepackage[english, activeacute]{babel}
\usepackage{amsmath,amsthm,amsxtra}
\usepackage{epsfig}
\usepackage{amssymb}
\usepackage{latexsym}
\usepackage{amsfonts}
\usepackage{hyperref}
\usepackage{pdftricks}
\newtheorem{thm}{Theorem}[section]

\newtheorem{lem}[thm]{Lemma}

\newtheorem{defn}[thm]{Definition}

\theoremstyle{remark}
\newtheorem{rem}{Remark}[section]

\newtheorem{Cl}{Claim}

\numberwithin{equation}{section}

\newcommand{\al}{\alpha}

\def\bm{\left( \begin{array}{cc}}
\def\endm{\end{array}\right)}

 \providecommand{\abs}[1]{\lvert#1 \rvert}

\newcommand{\be}{\begin{equation}}
\newcommand{\ee}{\end{equation}}
\newcommand{\ba}{\left(\begin{array}{c}}
\newcommand{\ea}{\end{array}\right)}
\newcommand{\bea}{\begin{eqnarray}}
\newcommand{\eea}{\end{eqnarray}}
\newcommand{\bee}{\begin{eqnarray*}}
\newcommand{\eee}{\end{eqnarray*}}
\newcommand{\ben}{\begin{enumerate}}
\newcommand{\een}{\end{enumerate}}

\begin{document}
\pagenumbering{arabic}	
\title[Stability in Half-line]{STABILITY OF KDV SOLITONS ON THE HALF-LINE}

\author{M\'arcio CAVALCANTE}
\address{\emph{Instituto de Matem\'atica, Universidade Federal de Alagoas,\\ Macei\'o-Brazil}}
\email{mcavalcante@dim.uchile.cl}
\let\thefootnote\relax
\author{Claudio Mu\~noz}
\address{\emph{CNRS and Departamento de Ingenier\'ia Matem\'atica, Universidad de Chile\\ Santiago, Chile}}
\email{cmunoz@dim.uchile.cl}

\begin{abstract}
	In this paper we study the stability problem for \emph{KdV solitons} on the left and right half-lines. Unlike standard KdV, these are not exact solutions to the equations posed on the half-line, and, contrary to NLS, no exact soliton solution seems to exist. However, we are able to show that solitons posed initially far away from the origin are strongly stable for the problem posed on the right half-line, assuming homogeneous boundary conditions. For the problem posed on the left half-line, the positive infinite-time stability problem makes no sense for the case of KdV solitons, but in this setting we prove a result of stability for all negative times. The proof involves finding and using two almost conserved quantities adapted to the evolution of the KdV soliton in the particular case of the half-line. Adaptations to other boundary conditions or star graphs are also discussed. 
	

\end{abstract}
\maketitle

\section{Introduction}
\medskip

\subsection{Setting of the problem} The Korteweg-de Vries (KdV) \cite{KDV} equation
\begin{equation}\label{KdV}
\partial_tu+\partial_x(\partial_x^2 u+u^2)=0,
\end{equation}
where $u=u(x,t)$ is a real-valued function, and $(x,t)\in \R^2$,
was introduced in 1877 by Boussinesq \cite{bous} as a model for long waves propagating on a
shallow water surface. Boussinesq's main motivation was to derive an equation with a smooth
traveling wave solution modeling the wave observed in 1834 by Scott Russell \cite{SR}. KdV does indeed possess such traveling wave solutions, also known as solitons. The KdV soliton is given by the formula
\begin{equation}
u(x,t)=\widetilde{Q}_c(x-ct-x_0),
\end{equation}
where
\begin{equation}\label{soliton}
\widetilde{Q}_c(s)= \frac{3c}{2}\text{sech}^2\left(\frac{\sqrt{c}s}{2}\right),
\end{equation}
with $c>0$ as the propagation speed of the wave and $x_0$ an arbitrary constant.  The real-line soliton $\widetilde{Q}_c$ satisfies the following ``boundary value problem'' (BVP) on $\R$,
\begin{equation}\label{soliton2}
\begin{cases}
\widetilde{Q}_c''-c\widetilde{Q}_c+\widetilde{Q}_c^2=0,& x,t\in \R,\\
\lim_{\pm \infty}\widetilde{Q}=0,
\end{cases}
\end{equation}
and it is the unique positive $H^1$-solution of \eqref{soliton2} up to translations in space. 
%

\medskip

In a more dynamical context, the initial value problem (IVP) for KdV posed on the real axis,
\begin{equation}\label{IVPKDV}
	\begin{cases}
		\partial_tu+\partial_x(\partial_x^2 u+u^2) =0 ,& (x,t)\in\R\times\R,\\
		u(x,0)=u_0(x),                                   & x\in\R,
	\end{cases}
\end{equation}
 has been extensively studied in the last years. We refer to the reader for instance to the works by Kenig-Ponce-Vega \cite{KPV0,KPV,KPV1}, Bourgain \cite{B}, Colliander, Keel, Staffillani, Takaoka and Tao \cite{KKSTT}, Guo \cite{Guo} and Kishimoto \cite{Kishimoto}, and references therein. For a complete and detailed account, see also the monograph by Linares and Ponce \cite{Linares-Ponce}. In particular, global well-posedness (GWP) does hold for data as far as in the inhomogeneous Sobolev spaces $H^s$, where $s\geq -\frac34.$ Below this regularity, the flow map on the real line $\mathbb R$ is known to be not uniformly continuous \cite{KPV2}.

\medskip

From another point of view, many physical problems arises naturally as initial boundary value problems (IBVP), because of the local character of the corresponding phenomenon \cite{Zabusky}. However, the IBVP for the KdV equation has been considerably less studied than the corresponding IVP \eqref{IVPKDV}. For example, there are at least two interesting IBVP for KdV still in unbounded domains: the one posed on the right half-line, and a second one posed on the left portion of the line, which we consider in this work.

\subsection{Unbounded initial boundary value problems} The IBVP for the KdV equation posed on the {\bf right} half-line is the following: for $\mathbb R^+:= (0,+\infty)$ and \textcolor{black}{$T>0$}, find a solution $u$ to
\begin{equation}\label{IBVP}
\begin{cases}
\partial_tu+\partial_x(\partial_x^2 u+u^2)=0  ,& (x,t)\in \mathbb R^+\times(0,T),\\
u(x,0)=u_0(x),                                   & x\in \mathbb R^+,\\
u(0,t)=f(t),                                     & t\in(0,T),
\end{cases}
\end{equation}
while the IBVP for the KdV equation posed on the {\bf left} half-line is the following: for $\mathbb R^-:= (-\infty,0)$ and \textcolor{black}{$T>0$}, find a solution $u$ to 
\begin{equation}\label{IBVP_left}
\begin{cases}
\partial_tu+\partial_x(\partial_x^2 u+u^2)=0  ,& (x,t)\in\mathbb R^-\times(0,T),\\
u(x,0)=u_0(x),                                   & x\in\mathbb R^-,\\
u(0,t)=f(t),                                     & t\in(0,T),\\
\partial_x u(0,t)=f_1(t),                                   & t\in(0,T).\\
\end{cases}
\end{equation}
Both problems differ in the sense that the one on the left half-line needs an additional boundary condition \textcolor{black}{(see \cite{Holmer} and \cite{Deconinck})}, making this problem more challenging from almost every point of view. As an example, our results differ from \eqref{IBVP} to \eqref{IBVP_left}.  More in general, for IBVPs, an important issue, both from the mathematical and physical point of view, is the study of the effect of the boundary condition(s) at $x=0$ on the asymptotic behavior of the solution.

\begin{rem}\label{leftleft}
	One could also consider the left half-line problem for negative time \begin{equation}\label{IBVPleftleft}
	\begin{cases}
	\partial_tu+\partial_x(\partial_x^2 u+u^2)=0  ,& (x,t)\in \mathbb R^-\times(-T,0),\\
	u(x,0)=u_0(x),                                   & x\in \mathbb R^-,\\
	u(0,t)=f(t),                                     & t\in(-T,0),
	\end{cases}
	\end{equation}
	although this is actually identical to the right half-line problem with positive time \eqref{IBVP} by the transformation $u(x,t)=u(-x,-t)$ (after properly changing $f(t)$).
\end{rem}
\medskip

In the recent literature, the mathematical study of IBVPs \eqref{IBVP} and \eqref{IBVP_left} is usually considered in the following setting 
\begin{equation}\label{setting_A}
(u_0,f)\in H^{s}(\R^+)\times H^{\frac{s+1}{3}}(\R^+),
\end{equation}
or
\begin{equation}\label{setting_B}
(u_0,f,f_1)\in H^{s}(\R^-)\times H^{\frac{s+1}{3}}(\R^+)\times H^{\frac{s}{3}}(\R^+),
\end{equation}
respectively. These assumptions are in some sense sharp because of the following localized smoothing effect for the linear evolution \cite{KPV}
\begin{align*}
&\|\psi(t) e^{-t\partial_x^3}\phi(x)\|_{C\big(\mathbb{R}_x;\; H^{(k+1)/3}(\mathbb{R}_t)\big)}\leq c \|\phi\|_{H^k(\mathbb{R})},\\
\intertext{and}
&\|\psi(t) \partial_xe^{-t\partial_x^3}\phi(x)\|_{C\big(\mathbb{R}_x;\; H^{k/3}(\mathbb{R}_t)\big)}\leq c \|\phi\|_{H^k(\mathbb{R})},
\end{align*}
where $\psi(t)$ is a smooth cutoff function and $e^{-t\partial_x^3}$ denotes the linear homogeneous solution group on $\mathbb{R}$ associated to the linear  KdV equation. Therefore, in what follows we will certainly follows both settings \eqref{setting_A}-\eqref{setting_B}.

\medskip

The mathematical study of the IBVP \eqref{IBVP} began with the work of Ton \cite{Ton}. He showed existence and uniqueness by assuming that the initial datum $u_0$ is smooth and the boundary data is $f=0$. Later, Bona and Winther \cite{BW} considered \eqref{IBVP} and proved global existence and uniqueness solutions in $L_{loc}^{\infty}(\R^{+};H^4(\R^+))$, for data $u_0\in H^4(\R^+)$ and $f\in H_{loc}^{2}(\R^+)$. In \cite{BW2}, they continued the study of \eqref{IBVP} and proved continuous dependence. Next, Faminskii \cite{Fa} considered a generalization of IBVP \eqref{IBVP} and obtained well-posedness in weighted $H^1(\R^+)$ Sobolev spaces. After this work, Bona, Sun and Zhang \cite{BSZ1} obtained \emph{conditional} local well-posedness, in the sense that solutions are only known to be unique if they satisfy some additional auxiliary conditions. This was done for data $u_0\in H^s(\R^+)$ and
$f\in H^{\frac{s+1}{3}}(\R^+)$ with $s>\frac{3}{4}$. They also proved global well-posedness for $u_0\in H^s(\R^+)$ and
$f \in H^{\frac{3s+7}{12}}(\R^+)$, with $1\leq s\leq3$. 

\medskip

The fundamental contribution of Colliander and Kenig \cite{CK} introduced a more dispersive PDE approach for the generalized Korteweg-de Vries (gKdV) equation posed on $\R^+$, based on writing the original IBVP \eqref{IBVP} as a superposition of three initial value problems on $\R\times\R$. In particular, for KdV \eqref{IBVP} this result gives conditional local well-posedness in $L^2(\R^+)\times H^{\frac{1}{3}}(\R^+)$, in which solutions are only known to be unique if they satisfy additional auxiliary conditions.  By the same time, Colliander and Kenig derived a global a priori estimate and for a non-optimal  boundary condition $f\in H^{\frac{7}{12}}(\R^+)$, and a conditional global well-posedness  was obtained  for the case $s=0$.  Recently, the first author \cite{Cavalcante} (using some of the Colliander-Kenig techniques) showed conditional local well-posedness for the IVP associated to the KdV equation on a \emph{simple star graph} given by two positive half-lines and a negative half-line attached in a common vertex.

\medskip

Later, Faminskii \cite{Fa2} improved the global results of \cite{CK} and obtained global results by assuming more natural boundary conditions (See Theorem \ref{well} below for more details). The local well-posedness
of the IBVP \eqref{IBVP} above $s=-\frac{3}{4}$, which is the critical Sobolev exponent for the KdV initial value
problem, was obtained by Holmer \cite{Holmer} and Bona, Sun and Zhang \cite{BSZ2}. Surveys describing these results and others are  \cite{BSZ1} and \cite{Fokas4}. 

\medskip

As for the left half-line case, Holmer \cite{Holmer} obtained local well-posedness in $H^s(\R^+)$ for $s>-\frac{3}{4}$. Then, Faminskii \cite{Fa4} obtained global well-posedness in $H^s(\R^+)$ for $s\geq 0$ for boundary conditions assuming natural conditions (see also Theorem \ref{well-left} for more details). 

\medskip

Another point of view for \eqref{IBVP} and \eqref{IBVP_left} is given by using Inverse Scattering techniques. Fokas \cite{Fokas1} introduced a new approach known as the unified transform method (UTM), which provides a proper generalization of the Inverse Scattering Transform (IST) method for solving IBVPs. For example, it is mentioned in \cite{Fokas2} that, under suitable decay and smoothness assumptions, just as in the infinite-line setting, the solution on the right half-line should describe (for large times) a collection of (standard KdV) solitons traveling at constant speeds. These techniques were further improved in \cite{Fokas4}, where well-posedness is proven in Sobolev spaces using the UTM method.

\medskip

\subsection{Main results} Before stating our main results, we explain the notion of soliton that we will use. First of all, as far as we know, no exact canonical soliton solution is available for problems \eqref{IBVP} and \eqref{IBVP_left}, except for some very particular cases. For example, for a soliton $\widetilde Q_c= \widetilde Q_c(x-ct-x_0)$, and $x_0\in \R$ given, let us define the natural half-line soliton as  
\be\label{Qc}
Q_c= \begin{cases} \widetilde{Q}_c\big|_{\R^+},  \hbox{ right half-line case,}\\
\widetilde{Q}_c\big|_{\R^-}, \hbox{ left half-line case.}
\end{cases}
\ee
where $\widetilde{Q}_c$ is the classical soliton in $\R$ given by the formula \eqref{soliton}. Note that this definition induces a ``natural'' trace $f(t)=\widetilde Q_c(-ct-x_0)$ of $Q_c$ at $x=0$ in \eqref{IBVP} (see e.g. Fokas and Lennels \cite{Fokas3} for more details on this point of view). However, this trace assumption strongly depends on the original soliton itself, and because of some energetic conditions, we will need a more suitable boundary condition which will have important consequences on the stability property.

\medskip

Instead, we will adopt the following approach: given any standard KdV soliton \eqref{soliton}, restricted to the half-line as in \eqref{Qc}, and placed far enough from the corner $x=0$, we will show that this solution is stable in the energy space under perturbations that preserve the zero boundary condition.    

\begin{thm}[Stability for the right half-line]\label{teorema} Let $c>0$ be a given constant. There exist constants $\alpha_0,C_0,L_0>0$ such that, for all $0<\alpha<\alpha_0$, and all $L>L_0$, the following is satisfied. Assume that $u_0\in H^1(\R^+)$ is such that 
	\begin{equation}\label{Sta_right_1}
	\begin{aligned}
	u_0(x=0)= &~ 0,\\
	  \|u_0-Q_c(\cdot-L)\|_{H^1(\R^+)}<&~ \alpha.
	\end{aligned}
	\end{equation}
	Then the solution $u=u(x,t)$ for the IBVP \eqref{IBVP} with boundary data $u(0,t)=f(t)\equiv 0$, given by Theorem \ref{well}, satisfies the global estimate
	\begin{equation}\label{Sta_right_2}
	\sup_{t \geq 0}\|u(t)-Q_c(\cdot-\rho(t)-L)\|_{H^1(\R^+)}<C_0(\alpha+e^{-\sqrt{c}L}),
	\end{equation}
	for a $C^1$-function $\rho(t)\in \R$ satisfying
	\begin{equation}\label{Sta_right_3}
	\sup_{t\geq 0}|\rho'(t)-c|< CC_0\alpha,
	\end{equation}
	for some constant $C>0$.
\end{thm}

This result shows strong stability of KdV soliton for the IBVP \eqref{IBVP}, which was, as far as we know, an open problem, even in the zero boundary condition case. Note that in Theorem \ref{teorema} we do not prove the existence of a soliton solution for \eqref{IBVP}, instead we show that the KdV soliton is the natural candidate to be the standard object appearing in the long-time dynamics, even if it is not an exact solution of the problem. Moreover, from simple formal computations, we believe that no soliton-like solution exists for \eqref{IBVP} with zero boundary condition, but we have no rigorous proof of this fact.

\medskip

Our proof is in some sense more dynamical than variational, because the half-line border introduces some important restrictions on the dynamics itself, which need to be controlled separately by using a particular extension property, as well as local estimates of the mass and energy of the KdV soliton (see Lemma \ref{dissipative}), which is only an approximate solution of the problem. We recall that in a general setting, \eqref{IBVP} has no conserved quantities, but we are still able to find some almost conserved quantities. For that reason, estimate \eqref{Sta_right_2} in Theorem \ref{teorema} accounts how far the soliton is placed at the initial time. 

\medskip

At the more technical level, we follow the approach introduced by Martel, Merle and Tsai \cite{MMT} for the study of the stability of generalized KdV multi-solitons in the energy space. This approach is based in the introduction of suitable almost conserved quantities and monotonicity properties, which are of proper interest. For the \eqref{IBVP} case, we follow a simplified version described in \cite{Mu}, which deal with the case of soliton-like objects of dispersive problems with no exact soliton solution.  We will take the advantage of a hidden dissipative mechanism of the model introduced through
imposition of the homogeneous boundary condition at $x=0$, see Lemma \ref{dissipative} for more details.

\medskip

Let us mention that the stability of KdV and more general objects is a large research area lasting for the past thirty years and more. Bona, Souganidis and  Strauss \cite{bona1}, Weinstein \cite{Weinstein}, Grillakis, Shatah and Strauss \cite{GSS}, Martel, Merle and Tsai \cite{MMT} and many others are important references in the field. In the case of KdV multi-solitons, it has also been proved stability even for $L^2$ perturbations \cite{AMV}. For a simple introduction to subject, we also refer to the monograph by Angulo \cite{Angulo}, see also \cite{Munoz} for a short review.

\begin{rem}[On the zero boundary condition]
Note that the condition $u(x=0,t)=0$ is assumed because of several important reasons. First of all, from energetic considerations most conserved quantities require the same flux at both sides of the boundary (i.e. zero net flux), and therefore the condition $u(t) \in H^1(\R^+)$ naturally imposes the zero boundary condition at $x=0$. Another reason to impose this condition is the fact that the exact 2-soliton solution of KdV $U(x,t)$ composed of exactly two large solitons but well-separated at the initial time (one on the right half-line, the other one on the left half-line), is an example of nonzero boundary data for which the corresponding evolution on the right half-line is far from being one soliton and a small perturbation. This shows that $f(t)$ could not be ``arbitrary'' (not even small for arbitrarily large times, and not even integrable in time probably). However, we believe that the stability property does hold for any $f(t)$ sufficiently small, as far as one can control a second derivative in space, integrated in time. However, this control should require more regularity on the solution, and therefore higher order conserved quantities. 
\end{rem}


\begin{rem}[On the zero boundary condition, 2]
Note that one may think that $u(x=0,t)=0$ for all $t\geq 0$ could lead to an odd extension of \eqref{IBVP} to the real line, where we know that KdV has stable solitons. However, this approach fails because KdV does not preserve the oddness property.\footnote{Note that the well-known nonlinear Schr\"odinger equation, which has solitary waves, preserves this oddness condition, therefore the odd extension of a zero boundary data is trivial. Additionally, standard conserved quantities such as mass and energy are conserved under the assumption $u(0,t)=0$, a nice property unfortunately not shared by the KdV dynamics.} In Section \ref{Section_4} (Definition \ref{Extension_a}) we will introduce a new extension of $u(x,t)$ to the real line, which has, as far as we understand, no dynamical meaning, but only a variational purpose. In that sense, this extension seems to be the ``least energy extension'' for the problem, that is to say, it acts only at the linear spectral level.  
\end{rem}

\begin{rem}
The evolution of the solution $u(t)$ given in Theorem \ref{teorema}, backwards in time, is an interesting open problem. As far as we know, no stability condition should hold in this situation, because of the collision of the soliton against the zero boundary condition at $x=0$. For instance, take as an example the KdV soliton $\widetilde Q_c(x-ct)$ restricted to the positive half-line, and with boundary condition $\widetilde Q_c(-ct)$. Then the evolution in time $t<0$ of  $\widetilde Q_c$ clearly shows that the solution must disappear inside the region $x<0$, a fact that reveals that the total mass (the $L^2$ norm squared) of the solution is decreasing to zero. Consequently, no stability of the soliton should hold in this particular regime. This simple example reveals that for problems posed on half-lines, conservation of physical quantities (if any) are essential to rule out some strange phenomena, usually not observed in the case of the whole real line.
\end{rem}

\begin{rem}
We also believe that KdV multi-solitons should also be stable if they are well-decoupled and ordered at the initial time, just as in Martel, Merle and Tsai's work \cite{MMT}. This result will certainly add more evidence for the fact that the soliton resolution conjecture for \eqref{IBVP} should be played by KdV solitons at first.
\end{rem}

\medskip



	From Remark \ref{leftleft} we have that a similar result is valid for the IBVP \eqref{IBVPleftleft}. We state without proofs the following corollary.
	\begin{thm}[Stability for the left half-line and negative time]\label{teorema_left} Let $c>0$ be a given constant. There exist $\alpha_0,C_0,L_0>0$ such that, for all $0<\alpha<\alpha_0$, and all $L>L_0$, the following is satisfied. Assume that $u_0\in H^1(\R^-)$ is such that 
		\begin{equation}
			\begin{aligned}
			u_0(x=0)=&~ 0, \\
			\|u_0-Q_c(\cdot + L)\|_{H^1(\R^-)}< & ~ \alpha.
			\end{aligned}
			\end{equation}
			Then the solution $u(x,t)$ for the IBVP \eqref{IBVPleftleft}, with boundary data $u(0,t)=f(t)\equiv 0$ given by Remark \ref{leftleft}, satisfies the global estimate
			\begin{equation}
			\sup_{t \leq 0}\|u(t)-Q_c(\cdot-\rho(t)+L)\|_{H^1(\R^-)}<C_0(\alpha+e^{-\sqrt{c}L}),
			\end{equation}
			for a $C^1$-function $\rho(t)\in \R$ satisfying
			\begin{equation}
			\sup_{t\leq 0}|\rho'(t)-c|< CC_0\alpha,
			\end{equation}
			for some constant $C>0$.
		\end{thm}

Note that this result is only valid if the time is taken backwards, in which case the soliton moves to the left, away from the corner $x=0$. We also recall that the stability problem forward in time makes no sense (at least for unbounded intervals of time) because of the collision between the soliton and the $x=0$ corner.

\medskip

From previous results above we have strong evidence that KdV solitons are part of the long time evolution in \eqref{IBVP} and \eqref{IBVP_left}. However, one may ask for \emph{the} natural exact solution of \eqref{IBVP} that replaces the (nonexact) soliton $Q_c$ and moves forward in time, as the usual KdV soliton does. This ``generalized soliton solution'' should be exact in the sense that solves the equation \eqref{IBVP}, and converges to a KdV soliton. The existence of this solution remains an interesting open problem. The main difficulty to prove such a result is the lack of $H^1$ control of the energy of the solution as the soliton solution approaches the $x=0$ corner. See Lemma \ref{dissipative} for more details.
\medskip

Our results can be recast as part of a series of recent works on the stability of solitonic structures on star graphs, most of them related to the nonlinear Schr\"odinger equation (NLS). For stability results, see the works of Adami, Cacciapuoti, Finco and Noja \cite{Noja1,Noja2}, and the recent papers \cite{Noja3} and \cite{Pava}. The stability for the logarithmic NLS equation on star graphs was obtained by Ardila \cite{Ardila}. A result of instability for the half-soliton stationary state for the NLS equation, with power nonlinearity, on a star graph, was obtained recently by  Kairzhan-Pelinovsky \cite{Pelinovsky}. See also \cite{Marzuola} for other recents developments in this area. Following \cite{Cavalcante}, we expect to deal with the corresponding \emph{KdV graph problem} in the near future. Additionally, the so-called \emph{Kirkchoff boundary condition} found in all above NLS papers, when taken in the KdV case, and to the limit case of only one half-line, naturally converges to the zero boundary condition used in this paper. However, let us emphasize that KdV on half-lines or simple nontrivial graphs, unlike NLS, has no conserved quantities unless we impose some very restrictive boundary conditions, which are not necessarily in agreement with the expected local well-posedness theory. This big difference is reflected, for instance, in the fact that no simple solitary waves seems to exist for KdV in half-lines. Finally, for the extension of the NLS fast soliton problem \cite{Holmer2} to the graph case, which is essentially a time-depending, long-time stability problem, see \cite{Adami}.

\subsection{Comments on the KdV equation on graphs}
	Let us expand further the previous comment. Consider the IBVP associated to the KdV equation on the simple star graph, given by two positive half-lines, i.e. 
	\begin{equation}\label{KDVgraph}
	\begin{cases}
	\partial_tu+\partial_x(\partial_x^2u+u^2)=0,& (x,t)\in(0,+\infty)\times (0,T),\\
		\partial_tv+\partial_x(\partial_x^2v+v^2)=0,& (x,t)\in(0,+\infty)\times (0,T),
	\end{cases}
	\end{equation}
	with initial conditions (see Fig. \ref{Fig1} below)
	\[
	u(x,0)=u_0(x),\;v(x,0)=v_0(x).
	\]
	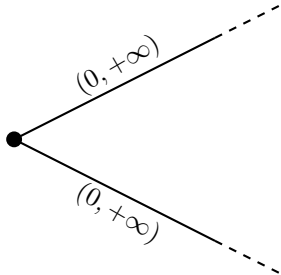
\begin{figure}[h!]
		\centering 
		\begin{tikzpicture}[scale=3]
		\draw[thick](0,0)--(0.89,0.45);
		\node at (0.46,0.33)[rotate=30]{$(0,+\infty)$};
		\draw [thick, dashed] (0.89,0.45)--(1.19,0.6);
		\draw[thick](0,0)--(0.89,-0.45);
		\node at (0.46,-0.33)[rotate=-30]{$(0,+\infty)$};
		\draw [thick, dashed] (0.89,-0.45)--(1.19,-0.6);
		\fill (0,0)  circle[radius=1pt];
		\end{tikzpicture}
		\caption{A star graph with two bonds}\label{Fig1}
	\end{figure}
	
	By direct application the ideas used in \cite{Cavalcante} and \cite{Holmer} we believe that the following two boundary conditions 
	\[
	u(0,t)=-v(0,t),
	\] 
	\[
	u_{xx}(0,t)=v_{xx}(0,t);
	\]
	are sufficient to ensure local well-posedness. On the other hand, formal computations of the mass and energy for IBVP \eqref{KDVgraph}, in the spirit of Lemma \ref{dissipative} below, suggest that the additional BC
%
\[
u_x(0,t)=v_x(0,t),
\]
must be imposed in order to ensure stability as in the proof of Theorem \ref{teorema}. Whether or not these joint three conditions ensure LWP is an open problem, to be discussed elsewhere. Note finally that the same problem for the KdV graph composed by a positive and negative half line is simply KdV on $\R$. Coincidentally, both mass and energy are conserved in this case.

\medskip

Lastly, we emphasize that most of these results can be readily extended to more general subcritical KdV models, provided a suitable Cauchy theory is at hand (in particular, a Cauchy theory dealing with high order traces of the solution at $x=0$, as in Theorem \ref{well}).

\subsection{Organization of this paper} This paper is organized as follows. In Section \ref{Section_2}, we recall the necessary methods and results about LWP and existence of trace estimates that we will need through this paper. Section \ref{Section_3} is devoted to establish some dispersive properties for the solutions of IBVPs \eqref{IBVP} \textcolor{black}{and \eqref{IBVP_left}}, involving the Mass and Energy of solutions. In Section \ref{Section_4} we prove Theorem \ref{teorema}.

\subsection{Acknowledgments}

M. C. wishes to thank the Centro de Modelamiento Matem\'atico (CMM) and Universidad de Chile and N\'ucleo Milenio CAPDE, for the financial support and nice scientific infrastructure that allowed to conclude the
paper during his postdoctoral stay. C.M. is partially funded by CMM Fondo Basal, Fondecyt No. 1150202, Nucleo Milenio CAPDE and MathAmSud EEQUADD 16Math-01. 

\section{Preliminaries}\label{Section_2}

In this Section we recall some well-known facts about the functional setting of \eqref{IBVP} and \eqref{IBVP_left}. For more details, the reader can consult Holmer's paper \cite{Holmer}. 

\subsection{Function Spaces}

For $s\geq 0$ we say that $\phi \in H^s(\mathbb{R}^+)$ if there exists $\tilde{\phi}\in H^s(\mathbb{R})$ such that 
$\phi=\tilde{\phi}|_{\R+}$.  In this case we set $\|\phi\|_{H^s(\mathbb{R}^+)}:=\inf\limits_{\tilde{\phi}}\|\tilde{\phi}\|_{H^{s}(\mathbb{R})}$. For $s\geq 0$ define 
$$
H_0^s(\mathbb{R}^+)=\Big\{\phi \in H^{s}(\mathbb{R}^+);\,\text{supp} (\phi) \subset[0,+\infty) \Big\}.
$$
 For $s<0$, define $H^s(\mathbb{R}^+)$ and $H_0^s(\mathbb{R}^+)$  as the dual space of $H_0^{-s}(\mathbb{R}^+)$ and  $H^{-s}(\mathbb{R}^+)$, respectively. Also, define 
$$
C_0^{\infty}(\mathbb{R}^+)=\Big\{\phi\in C^{\infty}(\mathbb{R});\, \text{supp}(\phi) \subset [0,+\infty)\Big\},
$$
and $C_{0,c}^{\infty}(\mathbb{R}^+)$ as those members of $C_0^{\infty}(\mathbb{R}^+)$ with compact support. We recall that $C_{0,c}^{\infty}(\mathbb{R}^+)$ is dense in $H_0^s(\mathbb{R}^+)$ for all $s\in \mathbb{R}$. 

\medskip

A definition for $H^s(\R^-)$ and $H_0^s(\R^-)$ is given analogously to that of $H^s(\R^+)$ and $H_0^s(\R^+)$. The following result summarize a pair of useful properties of restriction Sobolev spaces.



\begin{lemma}[\cite{CK}, Lemma 3.7]\label{alta}
	If $\frac{1}{2}<s<\frac{3}{2}$, the following statements are valid:
	\begin{enumerate}
		\item [(a)] $H_0^s(\R^+)=\big\{f\in H^s(\R^+);f(0)=0\big\},$\medskip
		\item [(b)] If  $f\in H^s(\R^+)$ with $f(0)=0$, then $\|\chi_{(0,+\infty)}f\|_{H_0^s(\R^+)}\leq c \|f\|_{H^s(\R^+)}$, where $\chi_{(0,+\infty)}$ denotes the characteristic function of $(0,+\infty)$.
	\end{enumerate}
\end{lemma}

\medskip

\subsection{Existence and continuity for the right half-line}

 For the purposes of this paper, we need some results of existence and continuity for the IBVPs \eqref{IBVP} and \eqref{IBVP_left}.
 
\begin{definition}
For any $T>0$ and $s\geq 0$, let $Z_T^s(\R^+)$ be the space given by the functions $u(x,t)$ satisfying
\begin{equation}\label{regularity}
\begin{split}
&\partial_t^m u\in C([0,T];H^{s-3m}(\R^+))\; ~\text{for any integer}\;~0\leq m\leq\frac{s}{3},\\
&\partial_x^l u\in C(\R^+;H^\frac{s-l+1}{3}(0,T))\; ~\text{for any integer}~\;0\leq l\leq s+1.
\end{split}
\end{equation}
\end{definition} 

A definition for $Z_T^s(\R^-)$ can be given analogous to that for $Z_T^s(\R^+)$.

 \begin{theorem}[Faminskii, \cite{Fa2}]\label{well}
 Consider the IBVP on the right half-line \eqref{IBVP}.  Fix a time $T>0$ as in \eqref{IBVP}. Let $u_0\in H^s(\R^+)$ and $f\in H^{\frac{s+1}{3}+\epsilon}(\R^+)$, such that the following conditions are satisfied:
 
 \ben
 \item The regularity $s\geq0$ is such that $\frac{s}{3}-\frac{1}{6}$ is not an integer,
 \item The parameter $\epsilon>0$ is arbitrary small in the case $s<1$, and $\epsilon$ can be taken equals zero in the case $s\geq1$. 
 \item The boundary datum $f$ satisfies the compatibility conditions 
 \[
 f^{(m)}(t=0)=\phi_m(x=0), \quad \hbox{for any integer ~ $0\leq m<\frac{s}{3}-\frac{1}{6}$,}
 \]
 where $\phi_0(x):=u_0(x)$ and for $0<m<\frac{s}{3}-\frac{1}{6},$
 	\begin{equation}\label{Compatibility}
\phi_m(x) :=-\phi_{m-1}'''(x)-\sum_{l=0}^{m-1} \binom{m-1}{l}\phi_l(x)\phi_{m-l-1}'(x).
 	\end{equation}
 \een
 Then there exists a solution $u(x,t)$ of the IBVP \eqref{IBVP} in the space $Z_{T}^s(\R^+)$. Moreover, the mapping $(u_0,f)\mapsto u$ is Lipshitz continuous on any ball in the norm of the mapping $H^s(\R^+)\times H^{\frac{s+1}{3}+\epsilon}(0,T)\rightarrow Z_{T}^s(\R^+)$.
 	\end{theorem}
 
\begin{rem}
The previous result is not only important in view of the GWP result, but also because of the compatibility conditions \eqref{Compatibility}, that lead to  \eqref{regularity}.
\end{rem}

\begin{rem}
The question of uniqueness in IBVPs posed on the half-line is relatively more complicated. In that sense Faminskii obtained (conditional) uniqueness in the sense introduced by Kato \cite{Kato}, i.e. in a subspace of $Z_T^s(\R^+)$. The unconditional uniqueness on the all space $C([0,T]:H^s(\R^+))$ was obtained by Bona, Sun and Zhang \cite{BSZ2}, by introducing the concept of mild solutions. The reader is referred to \cite{Bonaunconditional,Kato} for further discussions on unconditional and conditional well-posedness for general nonlinear evolution equations.
\end{rem}

 \subsection{Existence and continuity for the left half-line}

The following result of existence and continuity obtained by Faminskii \cite{Fa4} on the negative half-line its suffices for our purposes:
 \begin{theorem}\label{well-left}
 Consider the IBVP on the left half-line \eqref{IBVP_left}. Let $T>0$ be a fixed time, $u_0\in H^s(\R^-)$, $f\in H^{\frac{s+1}{3}+\epsilon}(\R^-)$ and $f_1\in H^{\frac{s}{3}+\epsilon}(\R^-)$. 
\ben
\item The regularity parameter $s\geq0$ is such that $\frac{s}{3}-\frac{1}{6}$ and $\frac{s}{3}-\frac{1}{2} $ are not integer, 
\item The parameter $\epsilon>0$ is arbitrary small in the case $s=0$, and $\epsilon=0$ in the case $s>0$. 
\item The boundary data $f_0$ and $f_1$ satisfy the compatibility conditions  
\[
f^{(m)}(t=0)=\phi_m(x=0),
\]
for any integer $0\leq m<\frac{s}{3}-\frac{1}{6}$, and 
\[
f_1^{(m)}(0)=\phi_m'(0),
\]
for any integer $0\leq m<\frac{s}{3}-\frac{1}{2}$; and where $\phi_0(x):=u_0(x)$, and 
\begin{equation}\label{Compatibility_2}
\phi_m(x):=-\phi_{m-1}'''(x)-\sum_{l=0}^{m-1} \binom{m-1}{l}\phi_l(x)\phi_{m-l-1}'(x).
\end{equation}
\een
 
 	Then there exists a solution $u=u(x,t)$ of the IBVP \eqref{IBVP_left} in the space $Z_T^{s}(\R^-)$. Moreover, the mapping $(u_0,f,f_1)\mapsto u$ is Lipshitz continuous on any ball in the norm of the mapping $H^s(\R^+)\times H^{\frac{s+1}{3}+\epsilon}(0,T)\times H^{\frac{s}{3}+\epsilon}(0,T)\rightarrow Z_{T}^s(\R^-)$.
 \end{theorem}

\bigskip

\section{Mass and Energy estimates}\label{Section_3}

\medskip

In this section we obtain several dispersive properties for the solutions of IBVPs \eqref{IBVP} and \eqref{IBVP_left}. These properties will involve suitable definitions of mass and energy. Unlike standard KdV on the line, in general mass and energy will not be conserved anymore, but under some additional assumptions, we will be able to prove that, even if they are not precisely conserved, at least they obey suitable estimates.

\subsection{The right half-line case} First we deal with the case of equation \eqref{IBVP}. This is certainly the less complicated case.
\begin{lemma}\label{dissipative}
Consider the following mass and energy functionals
\begin{equation}\label{Mass}
	M[u](t):=\frac12\int_0^{+\infty} u^2(x,t)dx,
\end{equation}
\begin{equation}\label{Energy}
E[u](t):=\int_0^{+\infty} \Big( \frac12 (\partial_x u)^2(x,t) -\frac13 u^3(x,t)\Big)dx,
\end{equation}
well-defined according to the initial conditions given. Then the solution $u=u(x,t)$ of the IBVP \eqref{IBVP} with 
\be\label{Homo_Cond}
u(0,t)=0 \quad \hbox{for all time $t\geq 0$,}
\ee
and initial datum $u_0 \in H^1(\R^+)$ satisfies, for all $t\geq 0$,
\begin{equation}\label{mass00}
M[u](t) \leq M[u_0],
\end{equation}
and
\begin{equation}\label{energy000}
E[u](t) \leq E[u_0].
\end{equation}
\end{lemma}

\begin{proof}
First, we assume that $u$ is sufficiently smooth and decays fast enough. A simple calculation shows that a smooth solution $u(x,t)$ of IBVP \eqref{IBVP} satisfies the identity
\be\label{Mass000}
\frac d{dt}\int_0^{+\infty} u^2(t) dx =  -\frac12 (\partial_xu)^2(0,t) +u(0,t)\Big(\partial_x^2u + \frac 23 u^2\Big)(0,t).
\ee
Indeed, multiplying the equation in \eqref{IBVP} by $u$ and integrating on $(0,\infty)$ in $x$ we get
	\begin{equation}
	\frac12 \frac{d}{dt}\int^{+\infty}_{0}u^2(x,t)dx=-\int^{+\infty}_{0}\partial_x(\partial_x^2u+u^2)udx.
	\end{equation}
	Integrating by parts, 
	\begin{equation}
	\begin{split}
	\frac12\frac{d}{dt}\int^{+\infty}_{0}u^2(x,t)dx&=\int^{+\infty}_{0}(\partial_x^2u+u^2)\partial_xudx-(\partial_x^2u+u^2)u\bigg|^{\infty}_0\\
	&=\int^{+\infty}_{0}\frac{d}{dx}\left(\frac12(\partial_xu)^2+\frac{u^3}{3}\right)dx \\
	& \quad + (\partial_x^2u(0,t)+u^2(0,t))u(0,t)\\
	&=-\frac 12 \partial_xu(0,t)^2- \frac13 u(0,t)^3\\
	& \quad  +(\partial_x^2u(0,t)+u^2(0,t))u(0,t)\\
	&= -\frac12(\partial_xu)^2(0,t) + u(0,t)\Big(\partial_x^2 u(0,t)+ \frac 23 u^2(0,t)\Big).
	\end{split}
	\end{equation}
We recall this last estimate again because it will be important for later purposes:
\be\label{dMass}
\frac12 \frac{d}{dt}\int^{+\infty}_{0}u^2(x,t)dx = -\frac12(\partial_xu)^2(0,t) + u(0,t)\Big(\partial_x^2 u(0,t)+ \frac 23 u^2(0,t)\Big).
\ee	
Note that, unless $u(0,t) =0$, we will have a source term in the mass coming from a second derivative term at $x=0$. This term is certainly very harmful and difficult to control by using only data in $H^1$. This fact certainly supports our choice of zero boundary condition on the corner $x=0$.

\medskip

Now, after integration in time in \eqref{Mass000}, we obtain
\begin{equation}\label{mass}
\begin{aligned}
\int_0^{+\infty} u^2(t) dx= &~ \int_0^{+\infty} u_0^2 dx-\frac12 \int_0^{t}(\partial_xu)^2(0,s)ds \\
&~ +\int_0^{t}u(0,s)\Big(\partial_x^2u + \frac 23 u^2\Big)(0,s)ds.
\end{aligned}
\end{equation}
Now we deal with the energy estimate. Indeed, we multiply the equation in \eqref{IBVP} by $(\partial_x^2u+u^2)$ and integrate on $(0,\infty)$ in $x$:
\[
\int_0^{+\infty}  \partial_t u ( \partial_x^2u+u^2) dx + \frac12\int_0^{\infty} \frac d{dx} (\partial_x^2u+u^2)^2  =0.
\]
Integrating by parts,
\[
\int_0^{+\infty} (- \partial_{tx}^2 u  \partial_x u+u^2 \partial_t u ) dx  - \partial_tu(0,t)\partial_xu(0,t) - \frac12 (\partial_x^2u+u^2)^2(0,t)  =0.
\]
Therefore, we obtain a new identity for the energy
\be\label{dEnergy}
\begin{aligned}
\frac{d}{dt}\int^{\infty}_{0}\left(\frac12 (\partial_xu)^2- \frac 13u^3\right)dx  = &~{} - \partial_tu(0,t)\partial_xu(0,t) \\
& ~{} - \frac12(\partial_x^2u(0,t)+u^2(0,t))^2.
\end{aligned}
\ee
This identity essentially says that, unless $\partial_tu(0,t) =0$, then the energy $E[u]$ has no definite dynamics. Once again, controlling the term $\partial_tu(0,t)$ in \eqref{dEnergy} is hard because it is related  through the original equation with third order derivatives in space.

\medskip

Consequently, replacing the equation \eqref{IBVP} and integrating in time,
\begin{equation}\label{energy}
\begin{split}
& \int_0^{+\infty} \Big( \frac12 (\partial_xu)^2 -\frac13 u^3\Big)(t)dx =\\
& \quad=\int_0^{+\infty} \Big( \frac12 (\partial_xu_{0})^2-\frac13 u_0^3\Big)dx \\
 &\qquad -\frac12 \int_0^{t}(\partial_x^2u(0,s)+u^2(0,s))^2ds \\
 & \qquad + \int_0^{t}\partial_xu(0,s) \partial_x(\partial_x^2 u(x,s) +u^2(x,s))\Big|_{x=0}ds.
  \end{split}
\end{equation}

Now we justify the last mass and energy computations for the case of $H^1$ data. Assume that $u(x,t)$ is a solution for the IBVP \eqref{IBVP} with initial data $u_0\in H^s(\R^+)$, with $u_0(0)=\partial_x^3u_0(0)=0$, and boundary data $f\in H^{\frac{s+1}{3}}(\R^+)$, for a given $s$ satisfying $\frac72<s<\frac{11}{2}$ given by Theorem \ref{well} (the condition for the third derivative of $u_0$ comes from the case $m=1$ in Theorem \ref{well}). From \eqref{regularity} we have that 
\begin{equation*}
\begin{aligned}
u(0,t)\in & ~H^{\frac{s+1}{3}}(\R^+), \\
\partial_xu(0,t)\in & ~H^{\frac{s}{3}}(\R^+),\\
\partial_{x}^2u(0,t)\in & ~H^{\frac{s-1}{3}}(\R^+),\\
 \text{ and }~ \partial_x^3u(0,t)\in & ~H^{\frac{s-2}{3}}(\R^+). 
\end{aligned}
\end{equation*}
It follows that for fixed $t$ we have that $\partial_tu(x,t)\in H^{s-3}(\R^+)$ has a well-defined trace at $x=0$. Hence for homogeneous boundary condition $u(0,t)=f(t)\equiv 0$ we have that 
\begin{equation}\label{trace}
0=\partial_tu(x,t)\big|_{x=0}=- \partial_x(\partial_x^2u(x,t) +u^2(x,t))\big|_{x=0}.
\end{equation}

Consequently, the identities for the mass \eqref{mass} and the energy \eqref{energy} take the form ($t\geq 0$)
\begin{equation}\label{mass0}
\int_0^{+\infty} u^2(t) dx = \int_0^{+\infty} u_0^2 dx-\int_0^t\frac12 (\partial_xu)^2(0,s) ds,
\end{equation}
\begin{equation}\label{energy0}
\begin{aligned}
\int_0^{+\infty} \Big( \frac12 (\partial_xu)^2 -\frac13 u^3\Big)dx = &~ \int_0^{+\infty} \Big( \frac12 (\partial_xu_{0})^2 -\frac13 u_0^3\Big)dx \\
&~  -\frac12\int_0^t (\partial_{xx}u(0,s)+u^2(0,s))^2.
\end{aligned}
\end{equation}
From \eqref{mass0} and \eqref{energy0} we have the following dissipative mechanism for the mass and the energy
\[
\int_0^{+\infty} u^2(t) dx \leq \int_0^{+\infty} u_0^2 dx,
\]
and
\[
\int_0^{+\infty} \Big( \frac12 (\partial_xu)^2 -\frac13 u^3\Big)dx \leq \int_0^{+\infty} \Big( \frac12 (\partial_xu_{0})^2 -\frac13 u_0^3\Big)dx.  
\]
Now assume $u_0\in H^1(\R^+)$. Let $\{u_{0n}\}$ be a bounded sequence in $H^{\frac72 ^{+}}(\R^+)$ such that $u_{0n}(0)=\partial_x^3u_{0n}(0)=0$ and
\begin{equation}
\|u_{0n}-u_0\|_{H^1(\R^+)}\rightarrow 0,\; \text{when}\; n\rightarrow +\infty.
\end{equation}
It follows from the previous analysis that the identities \eqref{mass00} and \eqref{energy000} are valid for all $u_n$. Now letting $n\rightarrow +\infty$ and using the continuity of flow data to solution given in Theorem \ref{well}  the result follows.
\end{proof}

\begin{remark}
Note that \eqref{mass0} and \eqref{energy0} can be recast as hidden trace smoothing effects for bounded in time solutions in the energy space. Indeed, under the boundary value condition $u(0,t)=0$ for all $t$, we have
\[
\int_0^t (\partial_xu)^2(0,s) ds \lesssim \sup_t \|u(t)\|_{L^2(0,\infty)}^2 ,
\]
and
\[
\int_0^t (\partial_{x}^2u)^2 (0,s)dx \lesssim  \sup_t \|u(t)\|_{H^1(0,\infty)}^2.
\]
\end{remark}
Now we deal with the left half-line case.


\subsection{The left half-line case}

Using similar arguments as in the proof of Lemma \ref{dissipative} we can obtain the following result for $u(t)$ placed in $\R^-$. Note that there are some important differences when considering general data.

\begin{lemma}\label{dissipativel}
	Consider now the following mass and energy functionals
	\begin{equation}\label{Mass_left}
	M[u](t):=\frac12\int_{-\infty}^0 u^2(t)dx,
	\end{equation}
	and
	\begin{equation}\label{Energy_left}
	E[u](t):=\int_{-\infty}^{0} \Big( \frac12 (\partial_xu)^2(t) -\frac13 u^3(t)\Big)dx,
	\end{equation}
	well-defined for $H^1(\R^-)$ solutions of the IBVP \eqref{IBVP_left} such that $u(t=0) =u_0 \in H^1(\R^-)$. Assume additionally the homogeneous boundary condition 
	\be\label{Assumptions0}
	u(0,t)=f(t)\equiv 0.
	\ee
	Then we have the following identities
	\begin{equation}\label{mass00left}
	M[u](t)\leq M[u_0], 
	\end{equation}
	and
	\begin{equation}\label{energy000right}
	E[u](t) \leq  E[u_0], \quad t\leq 0.  
	\end{equation}
\end{lemma}

\begin{rem}
Completely similar results can be obtained for solutions defined starting at time $t=t_0 \neq 0$. We only state the case $t=0$ for notational simplicity. 
\end{rem}

\begin{rem}
Conditions \eqref{Assumptions0}, although more restrictive than the right half-line case, are somehow natural (and necessary) for the dynamics on $\R^-$, even at the level of the Cauchy theory, see e.g. \eqref{IBVP_left}.  
\end{rem}

\begin{proof}
	We only give a formal proof by assuming that $u(x,t)$ is a smooth function and satisfies the decay properties when $x \rightarrow -\infty$. The general case where $u\in C((0,T);H^1(\R{^-}))$ can be obtained by an approximation argument as in the proof of Lemma \ref{dissipative}.
	
	\medskip
	
	Multiplying the equation in \eqref{IBVP_left} by $u$ we get
	\begin{equation}
	\frac12 \frac{d}{dt}\int_{-\infty}^{0}u^2(x,t)dx=-\int_{-\infty}^{0}\partial_x(\partial_x^2u+u^2)udx.
	\end{equation}
	Integrating by parts, 
	\begin{equation}
	\begin{split}
	\frac12\frac{d}{dt}\int_{-\infty}^{0}u^2(x,t)dx&=\int_{-\infty}^{0}(\partial_x^2u+u^2)\partial_xudx-(\partial_x^2u+u^2)u\bigg|_{-\infty}^0\\
	&=\int_{-\infty}^{0}\frac{d}{dx}\left(\frac12(\partial_xu)^2+\frac{u^3}{3}\right)dx \\
	& \quad -(\partial_x^2u(0,t)+u^2(0,t))u(0,t)\\
	&=\frac 12 \partial_xu(0,t)^2+ \frac13 u(0,t)^3 -(\partial_x^2u(0,t)+u^2(0,t))u(0,t)\\
	&= \frac12(\partial_xu)^2(0,t)-u(0,t)\Big(\partial_x^2 u(0,t)+ \frac 23 u^2(0,t)\Big).
	\end{split}
	\end{equation}
	Then, after integrating between $0$ and $t$, we obtain
	\begin{equation}\label{massleftt}
	\begin{aligned}
	M[u](t)= &~ M[u_0]+\int_0^{t}\frac12 (\partial_x u)^2(0,s)ds \\
	&~ {} -\int_0^{t}u(0,s)\Big(\partial_x^2u(0,s)+ \frac 23 u^2(0,s)\Big)ds.
	\end{aligned}
	\end{equation}
	Finally, note that under assumptions \eqref{Assumptions0} we obtain \eqref{mass00left}.

\medskip	
	
	To obtain \eqref{energyleftt} we multiply the equation in \eqref{IBVP_left} by $(\partial_x^2u+u^2)$ and integrate by parts,
	\begin{equation}
	\begin{split}
 &\frac{d}{dt}\int_{-\infty}^{0}\left(-\frac12 (\partial_xu)^2+\frac 13u^3\right)dx+\partial_tu(0,t)\partial_xu(0,t)=\\
 & \quad=\int_{-\infty}^{0}\partial_tu(\partial_x^2u+u^2)dx\\
&\quad =-\int_{-\infty}^{0}\partial_x(\partial_x^2u+u^2)(\partial_x^2u+u^2)dx\\
&\quad=-\int_{-\infty}^{0}\frac{d}{dx}\frac{1}{2}(\partial_x^2u+u^2)^2dx\\
&\quad=-\frac12(\partial_x^2u(0,t)+u^2(0,t))^2.
 \end{split}
 \end{equation}
 Therefore, integrating between 0 and $t$,
	\[
	\begin{split}
	E[u](t) =&~ E[u](0)+\frac12 \int_0^{t}(\partial_x^2u(0,s)+u^2(0,s))^2ds \\
	&+\int_0^{t}\partial_xu(0,s)(\partial_{x}^2u(x,s) +u^2(x,s))_x\big|_{x=0}ds.
	\end{split}
	\]
Finally, using \eqref{Assumptions0} we obtain
	\begin{equation}\label{energyleftt}
	\begin{split}
	E[u](t) =& ~ E[u](0)+\frac12 \int_0^{t}(\partial_x^2u(0,s))^2ds,
	\end{split}
	\end{equation}
which proves \eqref{energy000right}.
\end{proof}

\begin{rem}
Note that \eqref{energy000right} gives good uniform estimates in time in the case where $t<0$ only. For $t>0$ one has a boundary source term coming from the \emph{second derivative} of $u$ in $L^2$, which is a very bad term in the sense that one requires a third (or even fourth) order energy estimate to control it. 
\end{rem}

\bigskip

\section{Start of proof of Theorem \ref{teorema}: Extension to the entire line}\label{Section_4}

\medskip

The proof is based on the classical argument of Weinstein \cite{Weinstein}, with some minor changes coming from the fact that we do not work on the whole line, but only on $\mathbb R^+$, and the KdV soliton is not an exact solution of the problem by itself. See also \cite{Mu,Munoz} for a fully explained, similar argument. 

\medskip

The idea behind Weintein's result is to show a coercivity estimate, which is obtained using spectral properties of a well-chosen linear unbounded operator. In the next sections, we will find a suitable operator for the half-line case, to then extend it to the entire space to make use of the standard Weinstein's theory of stability.

\medskip

Take $c>0$ fixed and  $L>L_0>0$, where $L_0$ will be taken as large as needed. Assume that \eqref{Sta_right_1} is satisfied for an initial datum $u_0$, and for a certain $\al<\al_0$ to be chosen later. 

\medskip

Let $u(t)$ be the corresponding solution of  the IBVP \eqref{IBVP}, with boundary data $u(0,t)\equiv 0$, and given by Theorem \ref{well}. For $C_0>1$, consider the tubular neighborhood
\be\label{tubular}
\mathcal M[C_0]:= \Big\{ v\in H^1(\R^+) ~ : ~  \inf_{\rho_0\in \R} \|v-Q_c(\cdot-\rho_0 -L)\|_{H^1(\R^+)}<C_0(\alpha+e^{-\frac{\sqrt{c}}{2}L})\Big\}.
\ee
Note that from \eqref{Sta_right_1} we have $u_0 \in \mathcal M[C_0]$. We want to prove that for $L$ and $C_0$ large enough, and $\al<\al_0$ small, $u(t) \in \mathcal M[C_0]$ for al $t\geq 0$.
 
\medskip

Similarly, by the continuity of the KdV flow, we have $u(t) \in \mathcal M[C_0]$ for sufficiently small time $t$. Using a bootstrap argument, we will show the implication
\be\label{bootstrap}
t\geq 0, \quad u(t) \in \mathcal M[C_0] \implies u(t) \in \mathcal M[C_0/2], 
\ee
 which will prove \eqref{Sta_right_2}.

\subsection{Modulation} By taking $\al,L$ smaller, we can ensure the following decomposition argument:


\begin{lem}[Modulation]\label{Modulation}
Assume that  $u(t) \in \mathcal M[C_0]$ for all $t\geq 0$. Then, by taking $\al_0$ smaller and $L_0$ larger if necessary, there exists $\rho=\rho(t) \in \R$ such that we have the following decomposition:
\begin{equation}\label{Modulation_1}
u(x,t)=Q_c(x-\rho(t)-L)+z(x,t),
\end{equation}
where $z(x,t)$ satisfies, for all $t\geq 0$,
\be\label{Ortho}
\int_0^{+\infty} z(x,t) Q_c' (x-\rho(t)-L)dx =0,
\ee
and $\rho(t)$ satisfies the estimate
\be\label{rho_estimate}
|\rho'(t)-c| \lesssim \|z(t)\|_{H^1(\R^+)}  + e^{-\sqrt{c}L}.
\ee
Finally,
\be\label{z0}
\|z(0)\|_{H^1(\R^+)} \lesssim \alpha +e^{-\sqrt{c}L},
\ee
with an implicit constant independent of $C_0$.
\end{lem}

\begin{proof} The proof of this result is standard, but since we are working in the case of the half line, we need some small changes in the proof. See Appendix \ref{A} for a detailed proof.
\end{proof}

\begin{rem}
From \eqref{rho_estimate} we have the lower bound
\[
\rho(t) \geq \rho(0) + ct - tC_0(\alpha+e^{-\frac{\sqrt{c}}{2}L}),
\]
which for small $\alpha$ and large $L$ ensures that $\rho(t)$ is always an increasing function of time, $t\geq 0$. This fact will be used several times through the computations below. 
\end{rem}

\subsection{Extension to the whole line} The following step in the proof is a suitable extension of the spectral problem to the whole line. We will see that not every extension is useful, but a mild one will satisfy all the requirements.

\medskip

\begin{defn}[Zero extension, righ half-line case]\label{Extension_a}
Let $v\in H^1(\R^+)$ such that $v(x=0)=0$. We define its (zero) extension $\hat v$ as the function
\be\label{Extension_0}
\hat v(x):= \begin{cases} v(x) & x\geq 0 \\
0 & x<0.
  \end{cases} 
\ee
\end{defn}

\begin{rem}
Note that this extension makes sense in $H^1(\R^+)$, and gives a new function $\hat v\in H^1(\R)$ since  $v(x=0)=0$ (cf. Lemma \ref{alta}). Also, note that $\widetilde Q_c$ in \eqref{soliton} cannot be considered as the zero extension of $Q_c$ in \eqref{Qc}. This interesting difference will be important for the stability proof.
\end{rem}

We will apply the extension property to the function $u(t)$ in \eqref{Modulation_1}. More precisely, for each $t\geq 0$, let $\hat u = \hat u(x,t)$ be the zero extension function of  $u(t)$ defined using \eqref{Extension_0}. Also, recall $\widetilde{Q}_c(x-\rho(t)-L)$, the natural extension of $Q_c(x-\rho(t)-L)$ obtained reversing \eqref{Qc}. 

\medskip

For further purposes, let us define 
\be\label{tz}
\tilde{z}(x,t):=\hat{u}(x,t)-\widetilde{Q}_c(x-\rho(t)-L).
\ee
Note that both $\tilde z$ and $\widetilde{Q}_c(x-\rho(t)-L)$ obey somehow ``natural'' extensions, however $\hat u$ follows a completely different extension (by zero). More precisely, note that 
\be\label{tz_real}
\tilde{z}(x,t) = -\widetilde{Q}_c(x-\rho(t)-L), \quad x\leq 0.
\ee
We have the following useful set of estimates:
\begin{lem}\label{elementarlemma}
For $\tilde z(t)$ defined in \eqref{tz}-\eqref{tz_real} and $t\geq 0$, we have 
\be\label{tz_H1}
\hbox{$\tilde z (t) \in H^1(\R)$,}
\ee
as well as
\be\label{Est_tz}
\|\tilde{z}(t)\|_{H^1(\R^-)} +\|\tilde{z}(t)\|_{L^\infty(\R^-)} \lesssim e^{-\sqrt{c} |\rho(t) +L|}.
\ee
Finally, we have the global estimate
\be\label{Est_tz_2}
\|\tilde{z}(t)\|_{H^1(\R)}  \lesssim \|z(t)\|_{H^1(\R^+)} + e^{-\sqrt{c} |\rho(t) +L|},
\ee
with implicit constants independent of $t\geq 0$ and $C_0$.
\end{lem}
\begin{proof}
 Direct from \eqref{tz_real} and \eqref{soliton}.
\end{proof}

\bigskip

\section{Almost conserved quantities}

\medskip

Consider the decomposition of the dynamics \eqref{Modulation_1}. Under the condition $u(t) \in \mathcal M[C_0]$ in \eqref{bootstrap}, we know that $z(t)$ is a small perturbation of $Q_c$. Now we prove 

\begin{lem}[Energy and Mass expansions]\label{Expansions}
Recall the Mass $M[u]$ and Energy $E[u]$ defined in \eqref{Mass}-\eqref{Energy}. We have
\begin{equation}\label{eq3}
\begin{split}
M[u](t)=M[Q_c](t)+ \int_0^{+\infty}Q_c zdx+\frac{1}{2}\int_0^{+\infty}z^2dx,
\end{split}
\end{equation}
where $M[Q_c](t):= M[Q_c(\cdot -\rho(t) -L)]$ and $Q_c=Q_c(\cdot -\rho(t) -L)$. Similarly
\begin{equation}\label{eq1}
\begin{split}
E[u](t)=& ~ E[Q_c] (t)-\int_0^{+\infty}cQ_c zdx- Q_c(-\rho(t)-L)Q_c'(\rho(t)+L)\\
&+\frac12\int_0^{+\infty} (\partial_x z)^2dx-\int_0^{+\infty}Q_cz^2dx-\frac13\int_0^{+\infty}z^3dx.
\end{split}
\end{equation}
Here, $E[Q_c](t):= E[Q_c(\cdot -\rho(t) -L)]$. Finally, we have the following combined estimate:
\begin{equation}\label{tec1}
\begin{aligned}
& E[u](t)+cM[u](t)-E[Q_c](t)-cM[Q_c](t) = \\
& ~ \quad = O(e^{-2\sqrt{c}|\rho(t) +L|}) -\frac13\int_0^{+\infty}z^3dx\\
&~  \qquad +\frac{c}{2} \int_{0}^{+\infty} z^2dx+\frac12\int_{0}^{+\infty}(\partial_x z)^2dx\\
& ~ \qquad {}-\int_{0}^{+\infty} Q_c z^2dx.
\end{aligned}
\end{equation}
\end{lem}
\begin{proof}See Appendix \ref{Expansiones}.
\end{proof}

Let us continue with the proof. Let $\tilde{z}$ as in \eqref{tz}. We have from \eqref{tec1},
\[
\begin{aligned}
&~ E[u](t)+cM[u](t)-E[Q_c](t)-cM[Q_c](t)  =\\
& \quad = ~ O(e^{-2\sqrt{c}|\rho(t) +L|})-\frac13 \int_0^{+\infty} z^3dx\\
&~ \qquad +\frac{c}{2} \int_{\R}\tilde{z}^2dx + \frac12\int_{\R} (\partial_x \tilde z)^2dx-\int_{\R}\widetilde{Q}_c\tilde{z}^2dx \\
&~{} \qquad -\frac{c}{2} \int_{-\infty}^0\tilde{z}^2dx-\frac12\int_{-\infty}^0   (\partial_x \tilde z)^2 dx +\int_{-\infty}^0 \widetilde{Q}_c\tilde{z}^2dx.
\end{aligned}
\] 
Using \eqref{tz_real} and \eqref{Est_tz}, we have
\[
\Bigg| -\frac{c}{2} \int_{-\infty}^0\tilde{z}^2dx-\frac12\int_{-\infty}^0   (\partial_x \tilde z)^2 dx +\int_{-\infty}^0 \widetilde{Q}_c\tilde{z}^2dx \Bigg| \lesssim e^{-2\sqrt{c} |\rho(t)+L|}.
\]
Consequently,
\be\label{Step_medio}
\begin{aligned}
& E[u](t)+cM[u](t)-E[Q_c](t)-cM[Q_c](t)  =\\
& \quad = ~ O(e^{-2\sqrt{c}|\rho(t) +L|})-\frac13\int_0^{+\infty}
z^3dx\\
&~ \qquad +\frac{c}{2} \int_{\R}\tilde{z}^2dx + \frac12\int_{\R} (\partial_x \tilde z)^2dx-\int_{\R}\widetilde{Q}_c\tilde{z}^2dx.
\end{aligned}
\ee
Now we need some control on the terms $E[Q_c](t)$ and $M[Q_c](t)$. Note that these terms are not conserved in time (because $Q_c$ in \eqref{soliton} does not satisfy the zero boundary condition at $x=0$), and we cannot use the Lemma \ref{dissipative} and \eqref{mass00}. However, with some standard procedure we can get independent estimates.

\begin{lem}
Recall the terms $M[Q_c]$ and $E[Q_c]$ in \eqref{Mass}-\eqref{Energy} and $Q_c=Q_c(\cdot -\rho(t) -L)$. We have, for all $t\geq 0$,
\be\label{MQ_c_EQ_c}
\begin{aligned}
 M[Q_c](t) \gtrsim &~ M(Q_c)(0)-e^{-2 \sqrt{c}(L+\rho(t))} -e^{- 2\sqrt{c}(L+\rho(0))}. \\
 E[Q_c](t) \gtrsim &~ E(Q_c)(0)-e^{- 2\sqrt{c}(L+\rho(t))}-e^{- 2\sqrt{c}(L+\rho(0))}.
\end{aligned}
\ee
Moreover, for all $t\geq 0$,
\be\label{Conclusion}
\begin{aligned}
- \Big(E[Q_c](t) +c M[Q_c](t) \Big) \lesssim &~  -\Big( E(Q_c)(0) +c M[Q_c](0) \Big) \\
& + e^{- 2\sqrt{c}(L+\rho(t))}+ e^{- 2\sqrt{c}(L+\rho(0))}.
\end{aligned}
\ee
\end{lem}
\begin{proof}
We easily have  
\begin{equation*}
	\begin{split}
		M[Q_c](t)&=\frac 12 \int_{\R} \widetilde{Q}_c^2(x-\rho(t)-L)dx-\frac 12\int_{-\infty}^0\widetilde{Q}_c^2(x-\rho(t)-L)dx\\
		&=\frac 12 \int_{\R} \widetilde{Q}_c^2(x-\rho(0)-L)dx-\frac12 \int_{-\infty}^0\widetilde{Q}_c^2(x-\rho(t)-L)dx\\
		&=M[Q_c](0)+\frac 12 \int_{-\infty}^0 \tilde{Q}_c^2(x-\rho(0)-L)dx-\frac12 \int_{-\infty}^0\widetilde{Q}_c^2(x-\rho(t)-L)dx.
	\end{split}
\end{equation*}
Thus the first estimate follows by Lemma \ref{elementarlemma}.
On the other hand, $E[Q_c](t)$ can be easily estimated by
\begin{equation*}
\begin{split}
E[Q_c](t)&=\int_{\R} \Big( \frac12 \widetilde{Q}_c'^2(t) -\frac13 \widetilde{Q}_c^3(t)\Big)dx-\int_{-\infty}^0 \Big( \frac12 \widetilde{Q}_c'^2(t) -\frac13 \widetilde{Q}_c^3(t)\Big)dx\\
&=\int_{\R} \Big( \frac12 \widetilde{Q}_c'^2(t=0) -\frac13 \widetilde{Q}_c^3(t=0)\Big)dx-\int_{-\infty}^0 \Big( \frac12 \widetilde{Q}_c'^2(t) -\frac13 \widetilde{Q}_c^3(t)\Big)dx\\
&= E[Q_c](0)+\int_{-\infty}^0 \Big( \frac12 \widetilde{Q}_c'^2(t=0) -\frac13 \widetilde{Q}_c^3(t=0)\Big)dx \\
& \qquad -\int_{-\infty}^0 \Big( \frac12 \widetilde{Q}_c'^2(t) -\frac13 \widetilde{Q}_c^3(t)\Big)dx.
\end{split}
\end{equation*}
Recall $\widetilde Q_c = \widetilde Q_c(x-\rho(t) -L)$, where $\widetilde Q_c$ is given in \eqref{soliton}. It follows naturally from Lemma \ref{elementarlemma} that for each $t\geq 0$, 
\[
E(Q_c)(t)\gtrsim E(Q_c)(0)-e^{-2 \sqrt{c}(L+\rho(t))}-e^{-2 \sqrt{c}(L+\rho(0))}.
\]
This proves the last estimate in \eqref{MQ_c_EQ_c}.
%
\end{proof}

Now, combining \eqref{Step_medio} and \eqref{Conclusion},  we obtain that
\be\label{PaPa}
\begin{aligned}
E[u](t)+cM[u](t) = &~ \frac12 \Big( \int_{\R} (\partial_x \tilde z)^2dx + c \int_{\R}\tilde{z}^2dx -2\int_{\R}\widetilde{Q}_c\tilde{z}^2dx\Big)\\
&~  -\frac13\int_0^{+\infty} z^3dx -O(e^{-2 \sqrt{c}(\rho(t)+L)}+e^{-2 \sqrt{c}(\rho(0)+L)}).
\end{aligned}
\ee
On the other hand, Lemma \ref{dissipative} implies
\begin{equation}\label{mass01}
M[u](t)\leq M[u_0]\; \text{ and }\; E[u](t)\leq E[u_0],
\end{equation}
so that 
\[
E[u](t)+cM[u](t) - (E[u](0)+cM[u](0))\leq 0.
\]
Therefore, from this last inequality, \eqref{PaPa}, \eqref{Est_tz_2} and \eqref{z0},
\[
 \begin{aligned}
 \int_{\R} \Big( (\partial_x \tilde z)^2+ c \tilde{z}^2-2 \widetilde{Q}_c\tilde{z}^2\Big) dx \lesssim  &~   \|\tilde z(0)\|_{H^1(\R)}^2 +  \int_0^{+\infty} |z|^3dx \\
 & ~ +e^{-2 \sqrt{c}(L+\rho(t))}+e^{- 2\sqrt{c}(L+\rho(0))}\\
 \lesssim  &~  \|z(0)\|_{H^1(\R^+)}^2  + \|z(t)\|_{H^1(\R^+)}^3  \\
 &~ \quad  +e^{-2 \sqrt{c}(L+\rho(t))}+e^{- 2\sqrt{c}(L+\rho(0))} \\
 \lesssim  &~  \al^2 + \|  z(t)\|_{H^1(\R^+)}^3 \\
 &~ \quad + e^{-2 \sqrt{c}(L+\rho(t))}+e^{- 2\sqrt{c}(L+\rho(0))}.
 \end{aligned}
\]
Consequently,
\be\label{HH}
 \int_{\R} \Big( (\partial_x \tilde z)^2+ c \tilde{z}^2-2 \widetilde{Q}_c\tilde{z}^2\Big) dx \lesssim  \al^2 + \|  z(t)\|_{H^1(\R^+)}^3+e^{- 2\sqrt{c}L}.
\ee
The purpose of the next paragraph is to get a suitable lower bound on the term
\[
 \int_{\R} \Big( (\partial_x \tilde z)^2+ c \tilde{z}^2-2 \widetilde{Q}_c\tilde{z}^2\Big) dx = :  \int_{\R} \tilde z \mathcal L \tilde z dx,
\]
where 
\be\label{mL}
\mathcal{L}\tilde z:=-\partial_x^2 \tilde z+c \tilde z-2\widetilde{Q}_c\tilde z.
\ee
Note that we have reduced the problem on the half-line to an extended spectral problem on the whole line, and where we have good estimates on the left half-line portion of $\tilde z(t)$. 

\bigskip

\section{End of proof of Theorem \ref{teorema}} 

\medskip

Let us start out with the following easy estimate:
\begin{Cl}
We have, for all $t\geq 0$,
\be\label{Ayuda1}
\Bigg|\int_{-\infty}^0 \tilde z(x,t)\widetilde{Q}_c'(x-\rho(t)-L)dx\Bigg| \lesssim e^{-2\sqrt{c}|\rho(t) +L|}.
\ee
\end{Cl}
\begin{proof}
Direct from Lemma \ref{elementarlemma} and \eqref{Est_tz}. 
\end{proof}
From \cite{bona1,Weinstein} (see also \cite{Munoz} for more details), we have the standard coercivity estimate valid for each $\tilde z \in H^1(\R)$, $Q_c=Q_c(x-\rho(t)-L)$ and $\mathcal L$ as in \eqref{mL}:

\[
\begin{aligned}
\int_{\R} \tilde z \mathcal L \tilde z dx \gtrsim &~  \| \tilde z\|_{H^1(\R)}^2  - \Bigg|\int_{\R} \tilde z(x,t)\widetilde{Q}_c'(x-\rho(t)-L)dx\Bigg|^2 \\
&~  - \Bigg|\int_{\R} \tilde z(x,t)\widetilde Q_c(x-\rho(t)-L)dx\Bigg|^2.
\end{aligned}
\]
(Note that both quadratic reminder terms above are not zero in our case, but both are very small.) Using Lemma \ref{Modulation} and \eqref{Ayuda1} , we have
\[
\begin{aligned}
\int_{\R} \tilde z \mathcal L \tilde z dx \gtrsim &~  \| \tilde z\|_{H^1(\R)}^2   - \Bigg|\int_{\R} \tilde z(x,t)\widetilde{Q}_c(x-\rho(t)-L)dx\Bigg|^2 -e^{-2\sqrt{c}|\rho(t) +L|} \\
\gtrsim &~  \| z\|_{H^1(\R^+)}^2   - \Bigg|\int_{0}^{+\infty}  z(x,t)Q_c(x-\rho(t)-L)dx\Bigg|^2 -  e^{-2\sqrt{c}|\rho(t) +L|}.
\end{aligned}
\]
On the other hand we have
\begin{equation*}
\begin{split}
\int_0^{+\infty}z^2(x,t)dx=& \int_0^{+\infty}u^2(x,t)dx+M[Q_c](t)\\
&-2\int_0^{+\infty}z(x,t) Q_c(x-\rho(t)-L)dx.
\end{split}
\end{equation*}
Using this last expression, \eqref{mass00} and \eqref{MQ_c_EQ_c} we obtain
\begin{equation}\label{tec6}
 \Bigg|\int_{0}^{+\infty}  z(x,t)Q_c(x-\rho(t)-L)dx\Bigg| \lesssim \| z(0)\|_{H^1(\R^+)}+  \| z(t)\|_{H^1(\R^+)}^2 + e^{-2\sqrt{c}L}.
\end{equation}
Therefore, for $\al$ small and $L$ large,
\[
\int_{\R} \tilde z \mathcal L \tilde z dx \gtrsim  \| z\|_{H^1(\R^+)}^2  -  C(\al^2 + e^{-2\sqrt{c}L}).
\]
%
%
Combining this last estimate \eqref{tec1}  and \eqref{HH},  we obtain
\begin{equation}
\|z(t)\|_{H^1(\R^+)}^2 \lesssim \al^2 + e^{-2\sqrt{c}L} + \|z(t)\|_{H^1(\R^+)}^3,
\end{equation}	
with constants independent of $C_0$, which implies $u(t) \in \mathcal M[C_0/2]$ for $C_0,L$ large enough, and $\al$ small.

\medskip

\appendix

\section{Proof of Lemma \ref{Expansions}}\label{Expansiones} 

In this section we prove the expansions obtained in Lemma \ref{Expansions}.

\begin{proof}[Proof of Lemma \ref{Expansions}]
We compute: by definition of $z$ we see that
\begin{equation}\label{eq3a}
\begin{split}
cM[u](t)=cM[Q_c]+c\int_0^{+\infty}Q_czdx+\frac{c}{2}\int_0^{+\infty}z^2dx.
\end{split}
\end{equation}
On the other hand,
\begin{equation}\label{eq1a}
\begin{split}
E[u](t)=&E[Q_c]+\int_0^{+\infty}Q_c'z_xdx-\int_0^{\infty}Q_c^2zdx\\
&+\frac12\int_0^{+\infty}z_x^2dx-\int_0^{\infty}Q_cz^2dx-\frac13\int_0^{+\infty}z^3dx.
\end{split}
\end{equation}
Integrating by parts and using \eqref{soliton} we see that
\[
\begin{split}
\int_0^{+\infty}(Q_c'z_x-Q_c^2z)dx&=\int_0^{+\infty}(-Q_c''-Q_c^2)zdx-Q_c'(-\rho(t)-L)z(0,t)\\
&=\int_0^{+\infty}(-Q_c''-Q_c^2)zdx \\
&\quad + Q_c'(-\rho(t)-L)(f(t)-Q_c(-\rho(t)-L)),
\end{split}.
\]
therefore
\be\label{eq2}
\int_0^{+\infty}(Q_c'z_x-Q_c^2z)dx=-\int_0^{+\infty}cQ_czdx- Q_c(\rho(t)+L)Q_c'(\rho(t)+L).
\ee
Combining \eqref{eq1a}, \eqref{eq2} and \eqref{eq3a} we get
\begin{equation}\label{tec100}
\begin{aligned}
& E[u](t)+cM[u](t)-(E[Q_c](t)+cM[Q_c](t))=\\
&= ~ {}-Q_c(-\rho(t)-L)Q_c'(-\rho(t)-L)-\frac13\int_0^{+\infty}z^3dx\\
&~\qquad ~ +\frac{c}{2} \int_0^{+\infty}z^2dx+\frac12\int_0^{+\infty}z_x^2dx\\
&~\qquad ~ {}-\int_0^{+\infty}Q_cz^2dx.
\end{aligned}
\end{equation}
Note now that we easily have the pointwise estimate
\[
\abs{Q_c(-\rho(t)-L)Q_c'(-\rho(t)-L)} \leq C e^{-2\sqrt{c}|\rho(t) +L|}.
\]
%
Replacing this last estimate in \eqref{tec100}, we get the desired bound \eqref{tec1}. The proof is complete.
\end{proof}

\bigskip

\section{Proof of Lemma \ref{Modulation}}\label{A}

\medskip

The proof is standard. We have from \eqref{bootstrap} that for all $t\geq 0$,
\be\label{MaMa}
\inf_{\rho_0\in \R} \|u(t) -Q_c(\cdot-\rho_0 -L)\|_{H^1(\R^+)}<C_0(\alpha+e^{-L}).
\ee
if $C_0$ is large. Therefore, if we define the functional $\mathcal F =\mathcal F[v,\rho_0]$ by
\[
\begin{aligned}
& H^1(\R^+) \times (-L,\infty) \ni (v, \rho_0) \mapsto \\
& \qquad  \int_0^{+\infty} (v(x)- Q_c(x-\rho_0-L)) Q_c' (x-\rho_0-L)dx \in \R.
\end{aligned}
\]
it is not difficult to see that it is of class $C^1$ and $\mathcal F[Q_c(\cdot -\rho_0-L),\rho_0] =0$ for all $\rho_0\in (-L,\infty)$. Consequently, since
\[
\begin{aligned}
\frac{\partial}{\partial \rho_0} \mathcal F[v,\rho_0] \Big|_{v=Q_c(\cdot -\rho_0-L)} = &~  \int_0^{+\infty} Q_c'^2 (x-\rho_0-L)dx  \\
=&~ \int_{-\rho_0-L}^{+\infty} Q_c'^2  \geq \int_0^{+\infty}Q_c'^2 >0,
\end{aligned}
\]
by the Implicit Function Theorem we have that for all $v\in H^1(\R^+)$ such that $\|v -Q_c(\cdot-\rho_0 -L)\|_{H^1(\R^+)}<\delta_0$, there exists $\rho_0=\rho_0(v)>-L$ for which
\[
\int_0^{+\infty} (v(x)- Q_c(x-\rho_0-L)) Q_c' (x-\rho_0-L)dx =0.
\]
Using \eqref{MaMa} for small $\alpha$ and large $L$ if necessary, we have $C_0(\alpha+e^{-L})<\delta_0$, from which there exists $\rho(t)$ such that for all $t\geq 0$,
\[
\int_0^{+\infty} z(x,t) Q_c' (x-\rho(t)-L)dx =0, \quad  z(x,t) := u(x,t)- Q_c(x-\rho(t)-L). 
\]
The rest of the proof is standard, see e.g. \cite{MMT}.

\bigskip

%

\end{document}